\documentclass[12pt]{amsart}

\RequirePackage[colorlinks,citecolor=blue,urlcolor=blue]{hyperref}

\usepackage{amsmath,amstext,amssymb,amsopn,amsthm}
\usepackage{url,verbatim}
\usepackage{mathtools}
\usepackage{enumerate}

\usepackage{color,graphicx}

\usepackage[margin=30mm]{geometry}
\usepackage{eucal,mathrsfs,dsfont}%gives nicer set-names. 

\allowdisplaybreaks

\newtheorem{theorem}{Theorem}[section]
\newtheorem{corollary}[theorem]{Corollary}
\newtheorem{lemma}[theorem]{Lemma}

\theoremstyle{definition}

\newtheorem{definition}[theorem]{Definition}

\newtheorem{remark}[theorem]{Remark}

\numberwithin{equation}{section}

\newcommand{\eps}{\varepsilon}

\newcommand{\calT}{\mathcal{T}}

\newcommand{\calB}{\mathcal{B}}

\newcommand{\R}{\mathds{R}}

\newcommand{\e}{\varepsilon}

\newcommand{\wh}{\widehat}
\newcommand{\wt}{\widetilde}

\def\bv{{\bf v}}

\def\bx{{\bf x}}

\title[Hard ball collisions]{On the number of hard ball collisions}
\author{Krzysztof Burdzy and Mauricio Duarte}

\address{KB: Department of Mathematics, Box 354350, University of Washington, Seattle, WA 98195}
\email{burdzy@uw.edu}

\address{MD: Departamento de Matematicas, Facultad de Ciencias Exactas, Universidad Andres Bello, Santiago, Chile}
\email{mauricio.duarte@unab.cl}

\thanks{KB's research was supported in part by Simons Foundation Grant 506732. MD was supported by Programa Iniciativa Cientifica Milenio grant number NC120062 through the Nucleus Millenium Stochastic Models of Complex and Disordered Systems and FONDECYT project number 11160591}

\begin{document}

\begin{abstract}
We give a new and elementary proof   that the number of elastic collisions of a finite number of balls in the Euclidean space is finite. We  show that if there are $n$ balls of equal masses and radii 1, and 
 at the time of a collision between any two balls the distance between any other pair of balls is greater than $n^{-n}$,
then the total number of collisions is bounded by $n^{(5/2+\eps)n}$, for any fixed $\eps>0$ and large $n$. We also show that if there is a number of collisions larger than $n^{cn}$ for an appropriate $c>0$, then a large number of these collisions occur within a subfamily of balls that form a very tight configuration.
\end{abstract}

\maketitle

\section{Introduction}
\label{se:intro}

The purpose of this paper is to

(i)  give an ``elementary'' or ``conceptual'' proof of the claim  that the number of elastic collisions of a finite number of balls in the Euclidean space is finite, 

(ii) give a quantitative estimate for the time
when some subfamilies of the original family of balls stop to interact, 

(iii) give an explicit upper bound for the total number of collisions that is lower than the best known bound but requires extra assumptions, and

(iv) prove that if the number of collisions is very large then the balls have to form a tight configuration for an interval of time holding many collisions.

\subsection{Review of existing results}

The question of whether a finite system of hard balls in $\R^d$ can have an infinite number of elastic collisions was posed by Ya.~Sinai. It was answered in negative in \cite{Vaser79}. For alternative proofs see \cite{Illner89, Illner90,IllnerChen}. 
The papers \cite{BFK1,BFK2, BFK3,BFK4, BFK5} were the first to present universal bounds on the number of collisions of $n$ hard balls in any dimension.
It was proved in \cite{BFK1} that a system of $n$ balls in the Euclidean space undergoing elastic collisions can experience at most
\begin{align}\label{s26.1}
\left( 32 \sqrt{\frac{m_{\text{max}}}{m_{\text{min}}} } 
\frac{r_{\text{max}}}{r_{\text{min}}} n^{3/2}\right)^{n^2}
\end{align}
collisions. Here $m_{\text{max}}$ and $m_{\text{min}}$ denote the maximum and the minimum masses of the balls. Likewise, $r_{\text{max}}$ and $r_{\text{min}}$ denote the maximum and the minimum radii of the balls.
The  following alternative upper bound for the maximum number of collisions, not depending on the radii, appeared in \cite{BFK5},
\begin{align}\label{s26.2}
\left( 400 \frac{m_{\text{max}}}{m_{\text{min}}} 
 n^2\right)^{2n^4}.
\end{align}
  No improved universal bounds were found since then, as far as we know.
We conjecture that
the bounds in \eqref{s26.1}-\eqref{s26.2} are not  sharp. 

Let $K(n,d)$ be the maximum number of elastic collisions that $n$ balls in $\R^d$, of equal radii and masses, can undergo. It is easy to see that $K(n,1) =n(n-1)/2$ for $n\geq 2$, and that $K(n,d)$ is a non-decreasing function of $d$ (see \cite{BD}). Hence, $K(n,d) \geq n(n-1)/2$ for all $n\geq 2$ and $d\geq 1$.  Intuition may suggest that $K(n,d) =n(n-1)/2$ for every $d\geq 1$ because the balls are ``most constrained'' in one dimension; see \cite{MurCoh} for a historical review related to this point.
It turns out that this intuition is wrong. It is known that $K(3,2) = 4 > 3(3-1)/2$. An example showing that $K(3,2) \geq 4$ was found by J.D.~Foch and published in \cite{MurCoh}. The proof that $K(3,2) < 5$
was given in \cite{MC93}.​ 

It has been proved in \cite{BD} that if the balls have equal radii and masses then,
\begin{align}\label{d9.7}
K(n,d) \geq K(n,2) > n^3/27 \qquad \text{  for  } n\geq 3,\ d\geq 2.
\end{align}
There is a huge gap between the best known upper bounds in \eqref{s26.1}-\eqref{s26.2} and the best known lower bound in \eqref{d9.7}. This gap provides motivation for the present article.

\subsection{New results}

Our intention is to add some new methods to the existing techniques for proving finiteness of the number of collisions and for proving upper bounds for the number of collisions.
Papers
\cite{Vaser79,Illner89, Illner90,IllnerChen}
analyze a certain functional of the configuration but that analysis becomes qualitative at a certain point. We will develop a quantitative version of that method.
The authors of \cite{BFK1,BFK2, BFK3,BFK4, BFK5} translated the problem into the language of geometry of spaces with non-positive curvature (CAT(0) spaces).

First, in Theorem \ref{d10.2}, we will
 give an ``elementary'' or ``conceptual'' proof of the claim  that the number of elastic collisions of a finite number of balls in the Euclidean space is finite (the balls may have different masses and different radii). The new proof is based on simple properties of energy, momentum and elastic collisions so it is based more on physical intuition than on mathematical properties of the evolution. The proof contains almost no calculations. The main idea is that, for any fixed half-space, the component orthogonal to the half-space boundary of the total momentum of the  family of  balls that happen to be in the half-space at time $t$ is a monotone function of $t$. As a result, balls become ``ordered''  according to their velocities and some subfamilies of balls stop interacting. Our rigorous implementation of the idea departs somewhat from the above informal description.

Let the initial position of the $k$-th ball be denoted $x^k(0)\in\R^d$ and let $\bx(0)= (x^1(0), \dots, x^n(0)) \in \R^{nd}$.
Our second main result is the following.
\begin{theorem}\label{a8.4}
Consider $n$ balls of equal masses and radii 1.
Assume that the total momentum  of the balls is $0$ and their total energy is 1.
 The family of all $n$ balls can be partitioned into two non-empty subfamilies such that no ball from the first family will ever collide with a ball in the second family after time $100 n^3|\bx(0)|$.
\end{theorem}

The above result transforms some qualitative arguments originally developed in 
\cite{Vaser79,Illner89, Illner90} into a quantitative estimate. 
As a corollary, we will obtain the following.

\begin{theorem}
\label{th:nc}
Consider $n$ balls of equal masses and  radii  1.
Suppose that  at the time of a collision of any two balls, the distance between any other pair of balls is greater than $n^{-n}$.
Then, for any $\eps >0$, the total number of collisions is bounded by $n^{5n/2 + \eps n}$, for large $n$.
\end{theorem}

We will also use Theorem \ref{a8.4} to derive the following result. Suppose that a very large number of collisions occur. Then a smaller but also a very large number of collisions will have to occur in an interval of time during which a subset of the balls form a very tight configuration.
The main assumption of the theorem is that there are more than $n^{cn}$ collisions for an appropriate constant $c$. This assumption  might be void, that is, it is possible (we would even say likely) that the number of collisions is never that high. However, the theorem in the present form may be a precursor to a non-void  result, with  less stringent assumptions on the number of collisions. The theorem also gives moral support to the ``pinned billiards model'' investigated in a forthcoming paper \cite{ABD}.

Consider $n$ balls $B_k$ in $\R^d$ of equal masses and  radii  1, colliding elastically. 
For a fixed $\rho>0$,
let $\Gamma_\rho(t)$ be the graph whose vertices are balls $B_1, B_2, \dots, B_n$. Two vertices $B_j$ and $B_k$ are connected by an edge in $\Gamma_\rho(t)$  if and only if $|x^j(t)-x^k(t)|\leq 2+\rho$.

We will say that a subfamily $\{B_{i_1}, B_{i_2}, \dots, B_{i_k}\}$ of balls is $\rho$-connected in $[s,u]$ if for every $t \in [s,u]$, all balls $\{B_{i_1}, B_{i_2}, \dots, B_{i_k}\}$ belong to a  connected component of $\Gamma_\rho(t)$ (the connected component may depend on $t\in [s,u]$). 

\begin{theorem}
\label{th:Np}
Let $\rho \leq n^{-n}$ and $N>1$ be such that $\log(N\rho) > (3/2+\eps)n\log n$ for some $\eps>0$ independent of $n$. If the total number of collisions is greater than or equal to $N$ then there exist $n_0$, a family $\calB:=\{B_{i_1}, B_{i_2}, \dots, B_{i_k}\}$ of balls and an interval $[t_1,t_2]$ such that $\calB$ is $\rho$-connected in $[t_1,t_2]$ and there are more than $N\rho n^{-(3/2 + o(1))n}$ collisions among balls in $\calB$  on $[t_1,t_2]$, for $n\geq n_0$.
\end{theorem}

\begin{corollary}\label{a11.2}
If the total number of collisions is greater than or equal to $n^{\left(5/2+\eps\right)n}$ for some $\eps>0$ then there exist $n_0$, a family $\calB:=\{B_{i_1}, B_{i_2}, \dots, B_{i_k}\}$ of balls and an interval $[t_1,t_2]$ such that $\calB$ is $n^{- n}$-connected in $[t_1,t_2]$ and there are more than $n^{\eps n/2}$ collisions among balls in $\calB$  on $[t_1,t_2]$, for $n\geq n_0$.
\end{corollary}

\begin{corollary}\label{a11.1}
If the total number of collisions is greater than or equal to $n^{n^\alpha}$ for some $\alpha\in(1,2]$ then there exist $n_0$, a family $\calB:=\{B_{i_1}, B_{i_2}, \dots, B_{i_k}\}$ of balls and an interval $[t_1,t_2]$ such that $\calB$ is $n^{-\frac 1 3 n^\alpha}$-connected in $[t_1,t_2]$ and there are more than $n^{\frac 1 3 n^\alpha}$ collisions among balls in $\calB$  on $[t_1,t_2]$, for $n\geq n_0$.
\end{corollary}

The paper consists of five more sections. Section \ref{newproof} contains 
a new proof that the number of collisions is finite.
Section \ref{se:notation} collects notation and assumptions for the remaining part of the paper.
Section \ref{se:geometry} is devoted to the analysis of some functionals.
A new upper bound for the number of collisions is given in Section \ref{se:number}.
Section \ref{connect} contains the proof that balls form a tight configuration when a very large number of collisions occur.
An essential difference between Section \ref{newproof} and the sections following it is that the balls may have arbitrary masses and radii in Section \ref{newproof} but they are assumed to have identical masses and radii 1 in Sections \ref{se:notation}-\ref{connect}.

\section{The number of collisions  is finite}\label{newproof}

 We will consider $n\geq2$ hard balls $B_k$ in $\R^d$, for $d\geq 1$, colliding elastically. 
In this section, the balls may have different masses and different radii.

We say that a ``simultaneous collision'' occurs at time $t$ if there is a collection of balls $\{B_{i_1}, B_{i_2}, \dots, B_{i_k}\}$ for some $k\geq3$, such that for any two balls $B_{i_\ell}$ and $B_{i_p}$ in the family, there exist $j_1=i_\ell, j_2, \dots,j_{m-1}, j_m=i_p$ such that $B_{j_r}$ is in contact with $B_{j_{r+1}}$ at time $t$ for all $r=1, \dots, m-1$. For example, if a ball $B_1$ touches $B_2$ at time $t$ and balls $B_3$ and $B_4$ also touch at time $t$ but none of the balls from the first pair touches a ball from the second pair, we do not call $t$ a simultaneous collision time. This type of simultaneous occurrence of two collisions does not present any technical difficulties.

 It is known that the set of vectors in the phase space of positions and velocities that lead to simultaneous collisions has  measure zero (see \cite{Alex76}). 

If there are infinitely many collisions on a finite time interval then it is easy to see, using continuity of ball trajectories, that there exists a simultaneous collision. We will assume that there are no simultaneous collisions so the number of collisions will be finite on every finite time interval.

It should be pointed out that evolutions with simultaneous collisions are ``degenerate'' in the sense that the usual laws of physics (conservation of energy, momentum and angular momentum) do not uniquely determine the outgoing velocities (see, for example, \cite{Vaser79}). 

  We will assume  that the momentum of the system is zero. We can make this assumption because the number of collisions is the same in all inertial frames of reference. Since the total momentum is zero, the center of mass of all balls is constant, so it can and will be assumed to be at origin. 

Consider two moving balls $B_1$ and $B_2$ with centers $x^1(t)$ and $x^2(t)$, and velocities $v^1(t)$ and $v^2(t)$.
Suppose that the balls collide at time $t$. Let $P$ be the hyperplane tangent to the balls at the point of collision.
Let $v^1(t-) = \wt v^1(t-) + \wh v^1(t-)$, where $\wt v^1(t-) $ is orthogonal to $P$ and $\wh v^1(t-)$ is parallel to $P$. We decompose $v^2(t-)$ in an analogous way as $v^2(t-) = \wt v^2(t-) + \wh v^2(t-)$. The components $\wh v^1(t-)$ and $\wh v^2(t-)$ of the velocities will not change at the time $t$ of the collision.
Let $e_P = x^1(t) -x^2(t) $.
For the collision to take place we must have $(\wt v^1(t-)-\wt v^2(t-))\cdot e_P < 0$. It is easy to see that at the time of the collision  $\wt v^1\cdot e_P$ will have a positive jump and  $\wt v^2\cdot e_P$ will have a negative jump.

\begin{theorem}\label{d10.2}
(\cite{Vaser79})
The number of elastic collisions of a finite number of balls in the Euclidean space is finite, assuming no simultaneous collisions. 
\end{theorem}

\begin{proof}
If the total energy of the balls is zero, the balls are not moving and there will be no collisions. So let us assume that the total energy of the balls is strictly positive.

\emph{Step 1}.
 We will consider $n\geq2$ hard balls $B_k$ in $\R^d$, for $d\geq 1$, colliding elastically. The evolution will occur over the time interval $[0, \infty)$. 
The center and velocity of the $k$-th ball will be denoted $x^k(t)$ and $v^k(t)$, for $k=1,2,\dots,n$. We will  write $x^k(t) = (x^k_1(t), x^k_2(t), \dots, x^k_d(t))$ and 
$v^k(t) = (v^k_1(t), v^k_2(t), \dots, v^k_d(t))$.

We  define  ``order statistics'' $(y^1_j(t), y^2_j(t), \dots, y^n_j(t))$ for the centers of the balls in the direction of the $j$-th basis vector $e_j$ as the unique rearrangement of the numbers in the sequence $(x^1_j(t), x^2_j(t), \dots, x^n_j(t))$ such that $y^1_j(t)\leq y^2_j(t)\leq \dots \leq y^n_j(t)$. 

Fix some $j\in\{1,\dots,d\}$.
Let $\calT$ be the set of all times $t\geq 0$ such that there is a tie among $y^k_j(t)$'s, i.e., $y^k_j(t) = y^{k+1}_j(t)= \dots = y^r_j(t)$ for some $k$  and $r>k$.

The following is an implicit definition of $m(k,t)$.
 Let $B'_k(t)$ be the ball  $B_{m(k,t)}$ such that $x^{m(k,t)}_j(t) = y^k_j(t)$. If $t\in \calT$ then we label the balls $B'_k(t),\dots,B'_r(t)$ in such a way that  $m(k,t) < m(k+1,t) < \dots < m(r,t)$.

Let $w^k(t)= (w^k_1(t),\dots,w^k_d(t)) = v^{m(k,t)}(t)$, i.e., $w^k(t)$ denotes $B'_k$'s velocity at time $t$.

\bigskip
\emph{Step 2}.
We will show that 
$F^r(t) := \sum_{k=1}^r w^k_j(t)$ is a non-increasing function of $t$ for every $r=1, \dots, n$.

Suppose that $(t_1, t_2) \cap \calT = \emptyset$ 
and there are no collisions in $(t_1, t_2)$.
Then for every $k$, the function $t\to m(k,t)$ is constant on $(t_1, t_2)$ and so is the function $w^k_j(t)$. Hence $t\to \sum_{k=1}^r w^k_j(t)$ is
constant on $(t_1, t_2)$ for every $r$.

Consider a $t\in \calT$ such that no balls collide at $t$. First suppose that there are only two balls $B'_k(t)$ and $B'_{k+1}(t)$ such that $y^k_j(t) = y^{k+1}_j(t)$. Then there is $\eps>0$ such that there are no collisions in $(t,t+\eps)$ and $y^r_j(t) < y^{r+1}_j(t)$
for all $r\ne k$. If $w^k_j(t) \leq w^{k+1}_j(t)$ and since these velocities are not going to change in $(t,t+\eps)$, we will have $x^{m(k+1,t)}_j(s) \geq x^{m(k,t)}_j(s)$ for $s\in(t,t+\eps)$. Hence, we will have $w^k_j(s) \leq w^{k+1}_j(s)$ for $s\in(t,t+\eps)$. Similarly,  if $w^k_j(t) \geq w^{k+1}_j(t)$ then $x^{m(k+1,t)}_j(s) \leq x^{m(k,t)}_j(s)$ for $s\in(t,t+\eps)$. Once again, $w^k_j(s) \leq w^{k+1}_j(s)$ for $s\in(t,t+\eps)$.

We now use  time reversibility of the evolution to claim that $w^k_j(s) \geq w^{k+1}_j(s)$ for $s\in(t,t-\eps_1)$, for some $\eps_1>0$. Combining this with the claim that $w^k_j(s) \leq w^{k+1}_j(s)$ for $s\in(t,t+\eps)$, we conclude that $w_j^k(t+) = \min \left\{ w_j^k(t-), w^{k+1}(t-) \right\}$, and $w_j^{k+1}(t+) = \max \left\{ w_j^k(t-), w^{k+1}(t-) \right\}$. This shows that for $r\neq k$, the function $F^r(\cdot)$ is constant in a neighborhood of $t$, and that $F^k(
\cdot)$ is non-increasing on  $(t-\eps_1, t+\eps)$ for some $\eps,\eps_1>0$.

\bigskip
\emph{Step 3}.
Consider a $t\notin \calT$ such that some balls collide at $t$. First suppose that there are only two balls $B'_k(t)$ and $B'_r(t)$ that collide at time $t$. If $r=k+1$ and $y^k_j(t) = y^{k+1}_j(t)$ then the tangent plane $P$ at the collision point contains the basis vector $e_j$. The $j$-th components of the velocities of the two balls will not change at the moment of the collision and so the argument given in Step 2 applies and yields the same conclusion.

Next suppose that $y^k_j(t) < y^{r}_j(t)$. Let $w^k(t-) = \wt w^k(t-) + \wh w^k(t-)$, where $\wt w^k(t-) $ is orthogonal to $P$
and $\wh w^k(t-)$ is parallel to $P$.
We decompose $w^r(t-)$ in an analogous way as $w^r(t-) = \wt w^r(t-) + \wh w^r(t-)$. Then the components $\wh w^k(t-)$ and $\wh w^r(t-)$ of the velocities will not change at the time $t$ of the collision. 
Let $e_P$ be the unit vector orthogonal to $P$  such that $e_P \cdot e_j>0$.
For the collision to take place we must have $(\wt w^k(t-)-\wt w^r(t-))\cdot e_P > 0$ because we assumed that $y^k_j(t) < y^{r}_j(t)$ and, therefore, $(y^k_j(t) - y^{r}_j(t))\cdot e_P <0$. At the time of the collision  $\wt w^k(t)\cdot e_P$ will have a negative jump and  $\wt w^r(t)\cdot e_P$ will will have a positive jump, so the same applies to  $\wt w^k(t)\cdot e_j$  and  $\wt w^r(t)\cdot e_j$. 
This implies that $F^\ell(t)= \sum_{i=1}^\ell w^i_j(t)$ will have a non-positive jump at time $t$
for every $\ell$.

Extending the argument to the case when different pairs of balls have collisions at the same time, or the case when some collisions take place at a time $t\in \calT$, does not pose any conceptual problems so it is left to the reader.

\bigskip
\emph{Step 4}.
We have shown that 
$ \sum_{k=1}^r w^k_j(t)$ is a non-increasing function of $t$ for every $r=1, \dots, n$.

We apply the claim with $r=1$ to see that $ t\to w^1_j(t)$ is non-increasing. Hence, $w^1_j(\infty) := \lim_{t\to\infty} w^1_j(t)$ exists. The limit must be finite because all speeds are bounded since the energy of the system is constant.

Suppose that, for some $r< n$, we have shown that $w^k_j(\infty) := \lim_{t\to\infty} w^k_j(t)$ exists and is finite for  $k=1, \dots, r$. This and the fact that $ t\to\sum_{k=1}^{r+1} w^k_j(t)$ is  non-increasing imply that $w^{r+1}_j(\infty) := \lim_{t\to\infty} w^{r+1}_j(t)$ exists. The limit is finite for the same reason as in the case of $w^1_j(\infty)$.
By induction, we conclude that $w^k_j(\infty) := \lim_{t\to\infty} w^k_j(t)$ exists and is finite for all $k=1, \dots,n$.

We will argue that $w^k_j(\infty) \leq w^{k+1}_j(\infty)$ for all $k=1,\dots, n-1$. Suppose otherwise. Let $k$ be such that $w^k_j(\infty) > w^{k+1}_j(\infty)$.
The functions $t \to m(r,t)$ are not necessarily continuous but, despite that, the functions $t \to x^{m(r,t)}_j(t)$ are. The derivative of $t \to x^{m(r,t)}_j(t)$ exists and is equal to $w^r_j(t)$ for all except a countable number of $t$. The assumption that $w^k_j(\infty) > w^{k+1}_j(\infty)$ implies that, for some $\eps>0$,  $w^k_j(t) > w^{k+1}_j(t)+\eps$ for large $t$ and, therefore, $x^{m(k,t)}_j(t) > x^{m(k+1,t)}_j(t)+1$ for large $t$. This contradicts the definition of $x^{m(r,t)}_j(t)$ so the claim that $w^k_j(\infty) \leq w^{k+1}_j(\infty)$ for all $k=1,\dots, n-1$ has been proved.

Recall that we have assumed that the total momentum is zero and, therefore, the center of mass is not moving. This implies that we cannot have $ w^1_j(\infty) >0$ because then we would have $w^k_j(t) >0$ for all $k$ and sufficiently large $t$, contradicting the assumption that the center of mass is not moving. For a similar reason, we must have $w^n_j(\infty) \geq 0$.

It is possible that 
\begin{align}\label{d12.5}
w^1_j(\infty) =w^2_j(\infty) =\dots =w^n_j(\infty) = 0.
\end{align}
 For example, the centers of balls may move in the hyperplane orthogonal to $e_j$. But \eqref{d12.5} cannot hold for all $j=1,\dots,d$ simultaneously because that would imply that for all $j=1,\dots,d$,  every $\eps>0$, every $k=1,\dots, n$ and all sufficiently large $t$, we would have $|w^k_j(t)| < \eps$. This would contradict the assumption that the total energy is constant and strictly greater than zero.

We see that there must exist $j$ such that $w^1_j(\infty)  <w^n_j(\infty)$.
Thus for some $k$ and $j$, $w^k_j(\infty)  <w^{k+1}_j(\infty)$.

\bigskip
\emph{Step 5}.
We have shown that there exist $j,k$ and $\delta>0$ such that $w^k_j(\infty)+3\delta  <w^{k+1}_j(\infty)$. It follows that there exist $a\in \R$ and $s<\infty$ such that $w^r_j(t) +\delta <a < w^\ell_j(t)-\delta$ for all $r\leq k$, $\ell \geq k+1$ and $t\geq s$. 
We can and do change the inertial frame of reference so that $a=0$ and, hence,   $w^r_j(t) +\delta <0 < w^\ell_j(t)-\delta$ for all $r\leq k$, $\ell \geq k+1$ and $t\geq s$. 

Let $a_1, a_2>0$ be larger than the maximum of the radii of the balls.
Recall that the functions $t \to x^{m(i,t)}_j(t)$ are continuous
 for all $i$ and $j$. Moreover, their derivatives are equal to $w^i_j(t)$ for all except a countable number of $t$. It follows that for some $s_1< \infty$, we will have $x^{m(r,t)}_j(t)< - 2a_1$
and $x^{m(\ell,t)}_j(t)> 2a_2$ for all $r\leq k$, $\ell \geq k+1$ and $t\geq s_1$. 

This implies that for all $t\geq s_1$, some ball centers will satisfy $x^i_j(t)< -2a_1$, some of them will satisfy $x^i_j(t)> 2a_2$, and none of them will satisfy $-2a_1<x^i_j(t)< 2a_2$. The balls move continuously, so none of the balls $B_i$ with $x^i_j(s_1) <-2a_1$ will ever collide with a ball $B_\ell$ with $x^\ell_j(s_1) >2a_2$ after time $s_1$. Hence,  the original family of balls can be decomposed into two non-empty collections of balls with the property that no ball from one subfamily will collide with a ball in the other subfamily after time $s_1$.

\bigskip
\emph{Step 6}.
 Since the original family of balls can be decomposed into two non-empty collections of balls with the property that no ball from one subfamily will collide with a ball in the other subfamily after some time, the same reasoning can be applied to each of the subfamilies. Proceeding by induction, we will find a time $s_*<\infty$ such that we can decompose the family of $n$ balls into $n$ subfamilies of balls, each one containing only one ball, with the  property that no ball from one subfamily collides with a ball in the other subfamily after time $s_*$. Recall from the introduction to this section that there are no accumulation points in $[0,s_*]$ for the collision times. It follows that the number of collisions must be finite.
\end{proof}

\section{Upper bound: assumptions and notation}
\label{se:notation}

The following notation and assumptions will remain in force for the rest of the paper.

We will consider $n\geq 3$ hard balls in $\R^d$, for $d\geq 2$, colliding elastically, on the time interval $(-\infty, \infty)$.
If there are only two balls, they can collide at most once.

A crucial difference between Section \ref{newproof} and the remaining part of the paper is that from now on we will assume that the balls have equal masses and their radii are 1.

The center and velocity of the $k$-th ball will be denoted $x^k(t)$ and $v^k(t)$, for $k=1,2,\dots,n$. We will say that the $j$-th and $k$-th balls collide at time $t$ if $|x^j(t) - x^k(t)| = 2$ and their velocities change at this time.

The velocities are constant between collision times. The norm of velocity will be called speed (as is done in physics). 
We will write $\bx(t) = (x^1(t), \dots, x^n(t)) \in \R^{dn}$
and
$\bv(t) = (v^1(t), \dots, v^n(t))\in\R^{dn}$. Note that $\bv(t)$ is well defined only when $t$ is not a collision time, but both $\bv(t-)$ and $\bv(t+)$ are well defined for all times. 

Recall that all balls have the same mass. This implies that the velocities change at the moment of collision as follows. Suppose that the $j$-th and $k$-th balls collide at time $t$. 
This implies that the velocities $v^j(t-) $ and $ v^k(t-)$ (i.e., the velocities just before the collision) satisfy
\begin{align}\label{oc4.3}
(v^j(t-) - v^k(t-)) \cdot (x^j(t) - x^k(t)) < 0.
\end{align}
Let $x^{jk}(t) = (x^j(t) - x^k(t))/|x^j(t) - x^k(t)|$.
Then the velocities just after the collision are given by
\begin{align}\label{oc2.3}
v^j(t+) &= v^j(t-) + (v^k(t-) \cdot x^{jk}(t)) x^{jk}(t)
- (v^j(t-) \cdot x^{jk}(t)) x^{jk}(t),\\
v^k(t+) &= v^k(t-) + (v^j(t-) \cdot x^{jk}(t)) x^{jk}(t)
- (v^k(t-) \cdot x^{jk}(t)) x^{jk}(t).\label{oc4.2}
\end{align}
In other words, the balls exchange the components of their velocities that are parallel to the line through their centers at the moment of impact. The other components of velocities remain unchanged. It easily follows form \eqref{oc2.3}, \eqref{oc4.2}, and \eqref{oc4.3} that
\begin{align}
\label{oc3.5}
\bx(t)\cdot(\bv(t+)-\bv(t-)) &= - (v^j(t-) - v^k(t-)) \cdot (x^j(t) - x^k(t)) >0,
\end{align}
that is, $\bx(t)\cdot\bv(t+)>\bx(t)\cdot\bv(t-)$ when there is a collision at time $t$. 

\bigskip

We will make the following assumptions.

(A1)
The balls have equal masses and all radii are equal to 1.

(A2)
We will assume that there are no simultaneous collisions. See the beginning of Section \ref{newproof} for the discussion of this assumption.

(A3)  We will assume  that the momentum of the system is zero, i.e., $\sum_{j=1}^n v^j(t) =0$ for all $t$. We can make this assumption because the number of collisions is the same in all inertial frames of reference. Since the total momentum is zero, the center of mass of all balls is constant, so it can and will be assumed to be at the origin. This, together with the fact that all masses are equal, implies that $\sum_{j=1}^n x^j(t) =0$.  

(A4) We will assume without loss of generality that the total  ``energy'' is equal to 1, i.e., $|\bv(t+)|^2 =1$ for all $t$. If the total initial energy is equal to zero then the balls are not moving and there will be no collisions. If the initial energy is not zero then we can multiply all velocity vectors  by the same scalar constant so that the energy is equal to 1.  If all velocities are changed by the same multiplicative constant then the balls will follow the same trajectories at a different rate and hence there will be the same total number of collisions.

(A5) The problem of the number of collisions is invariant under time shifts. In Remark \ref{re:zero}, we will choose a specific time in the evolution process to play the role of time 0.

\section{Functionals of motion}
\label{se:geometry}

In this section we analyze  $\bx(t)$ and 
$\bv(t)$. Some of our proofs are based on ideas originally developed in \cite{Vaser79,Illner89, Illner90}.

\begin{definition}
For $u\in \R$, let
\begin{align*}%\label{de:x_u}
\bx_u(t) = \begin{cases}
\bx(t) & \text{  for  } t< u, \\
\bx(u) + (t-u)\bv(u+) & \text{  for  } t\geq u,
\end{cases}
\end{align*}
and let $\bv_u(t+)$ be the right derivative of $\bx_u(t)$.
\end{definition}

Let $T_C$ denote the set of collision times. We will use $\angle$ to denote the unsigned angle between vectors, i.e., 
\begin{align*}
\angle(w_1,w_2)
=\arccos \left(\frac{w_1}{|w_1|} \cdot \frac{w_2}{|w_2|}\right) \in[0,\pi]
\end{align*}
for any non-zero vectors $w_1, w_2$.

\begin{lemma}
\label{le:angle_dwn}
Consider any $u\in\R$. 
The function $\alpha(t):= \angle(\bx_u(t),\bv_u(t+))$ is strictly decreasing.
The function $t\to \angle(\bx(t),\bv(t+))$ is  also strictly decreasing.
\end{lemma}
\begin{proof}
Let $s<t$ be times in $\R\setminus T_C$. We will show that $\alpha(s)>\alpha(t)$ by induction on the number of collisions in $(s,t)$. Assume there are no collisions in $[s-\e,t+\e]$ for some $\e>0$. It follows that $\bv_u(r)=\bv_u(s)$  for all $r\in[s,t]$. 
Consider the triangle with vertices $0$, $\bx_u(s)$ and $\bx_u(t)$. It is easy to see that $\angle (\bx_u(s) - 0 , \bx_u(t) - \bx_u(s))
> \angle (\bx_u(t) - 0 , \bx_u(t) - \bx_u(s))$. This implies that
\begin{align*}
\alpha(s) 
&=\angle(\bx_u(s),\bv_u(s)) 
= \angle(\bx_u(s)-0,(t-s)\bv_u(s)) 
= \angle (\bx_u(s) - 0 , \bx_u(t) - \bx_u(s))\\
&> \angle (\bx_u(t) - 0 , \bx_u(t) - \bx_u(s))
=\angle(\bx_u(t),(t-s)\bv_u(t)) 
=\angle(\bx_u(t),\bv_u(t))\\
& =\alpha(t).
\end{align*}

Assume that the claim is true when there are $j$  collisions in $(s,t)$. If there are $j+1$ collisions, and $t_*$ is the first collision in $(s,t)$, then, by the first part of the proof, $\alpha(t_*-)\leq\alpha(s)$. Formula \eqref{oc3.5} implies that $\alpha(t_*+)<\alpha(t_*-)$. It follows from this and the induction assumption  that $\alpha(t_*-)>\alpha(t_*+)\geq\alpha(t)$. We conclude that $\alpha(s)>\alpha(t)$.

Suppose that $s<t$ are arbitrary, i.e., $s,t\in\R$, and they may belong to $ T_C$.
Then the inequality  $\alpha(s) >\alpha(t)$ follows from the previous case and right continuity of $\alpha(\cdot)$. 

If  $u$ is greater than the time of the last collision then  $\bx_{u}\equiv\bx$, and the second claim of the lemma follows from the first one.
\end{proof}

\begin{remark}
\label{re:zero}
Let $\ell$ be the time of the last collision. For $t>\ell$ we have $\bx(t)\cdot\bv(t) = \bx(\ell)\cdot\bv(\ell+) + (t-\ell)$, from which it easily follows that $\lim_{t\to\infty} \alpha(t)=0$. By time reversal, we have $\lim_{t\to -\infty} \alpha(t)= \pi$. This, and the previous lemma show that there is a unique $t_0\in\R$ such that $\alpha(t) >\pi/2$ for $t<t_0$, and $\alpha(t)<\pi/2$ for $t>t_0$. Right continuity yields $\alpha(t_0)\leq\pi/2$. Since the total number of collisions of a family of balls is invariant under time translation, we can and will assume from now on that $t_0=0$.
\end{remark}

\begin{lemma}
\label{le:angle_cut}
For all times $0\leq u < w$ and $t\in\R$ we have,
\begin{align}
\label{eq:angle_cut} \angle\left( \bx_w(t),\bv_w(t+)\right) 
&\leq \angle\left( \bx_u(t),\bv_u(t+)\right),\\
\angle\left( \bx(t),\bv(t+)\right) 
&\leq \angle\left( \bx_u(t),\bv_u(t+)\right).\label{d11.1}
\end{align}
\end{lemma}
\begin{proof}

We will proceed by induction on the number of collisions in the interval $(u,w]$. If no collisions occur in this interval then $\bx_w(t)=\bx_u(t)$ for all $t\in\R$, and \eqref{eq:angle_cut} is obviously true.

Assume that \eqref{eq:angle_cut} is true when $k$ collisions occur in $(u,w]$. If there are $k+1$ collisions in $(u,w]$, say, at times $u<t_1<t_2 <\cdots<t_{k+1}\leq w$, then there are  only $k$ collisions  in $(t_1,w]$ so, by the induction assumption, for all $t$,
\begin{align*}
\angle\left(  \bx_w(t),\bv_w(t+) \right) \leq \angle\left(  \bx_{t_1}(t),\bv_{t_1}(t+) \right).
\end{align*}
To finish the proof of \eqref{eq:angle_cut}, it will suffice to show that 
\begin{align}\label{d11.2}
\angle\left(  \bx_{t_1}(t),\bv_{t_1}(t+) \right)\leq  \angle\left(  \bx_{u}(t),\bv_{u}(t+) \right)
\end{align}
for all $t$. 
For $t < t_1$ we have $\bx_{t_1}(t)=\bx_{u}(t)$, hence \eqref{d11.2} is satisfied. 

Suppose that $t\geq t_1$. Note that $\bx_{t_1}(t_1) = \bx_{u}(t_1)=\bx(t_1)$. We have
\begin{align}\label{a3.1}
 \angle\left(  \bx_{t_1}(t),\bv_{t_1}(t+) \right) &=  \angle\left(  \bx(t_1)+(t-t_1)\bv(t_1+),\bv(t_1+) \right), \\
  \angle\left(  \bx_{u}(t),\bv_{u}(t+) \right) &=  \angle\left(  \bx(t_1)+(t-t_1)\bv(t_1-),\bv(t_1-) \right).\label{a3.2}
  \end{align}
Let
\begin{align*}  
\beta(\lambda) &= \arccos\left( \frac{\lambda + (t-t_1) }{\sqrt{|\bx(t_1)|^2 +2(t-t_1)\lambda+(t-t_1)^2}} \right).
\end{align*}
It follows from \eqref{a3.1}-\eqref{a3.2} that
\begin{align}\label{a3.3}
 \angle\left(  \bx_{t_1}(t),\bv_{t_1}(t+) \right) &= \beta( \bx(t_1)\cdot\bv(t_1+) ),
 \\ 
\angle\left(  \bx_{u}(t),\bv_{u}(t+) \right) & =\beta(  \bx(t_1)\cdot\bv(t_1-) ) .\label{a3.4}
\end{align}

Differentiation of $\beta(\lambda)$ shows that it is a decreasing function of $\lambda$ for $\lambda\geq 0$ and $t\geq t_1$. By  Lemma \ref{le:angle_dwn}, Remark \ref{re:zero} and since $u\geq 0$, we have
$$
\angle(\bx(t_1),\bv(t_1+)) \leq \angle(\bx(t_1),\bv(t_1-)) \leq  \angle(\bx(u),\bv(u+)) \leq \pi/2.
$$
This and \eqref{a3.3}-\eqref{a3.4} imply that, for $t\geq t_1$,
$$
 \angle\left(  \bx_{t_1}(t),\bv_{t_1}(t+) \right) = \beta( \bx(t_1)\cdot\bv(t_1+) ) \leq \beta(  \bx(t_1)\cdot\bv(t_1-) ) =   \angle\left(  \bx_{u}(t),\bv_{u}(t+) \right).
$$
This completes the proof of \eqref{d11.2} and, consequently, proves \eqref{eq:angle_cut}.

If we take $w$ greater than the time of the last collision then  $\bx_{w}\equiv\bx$, so \eqref{d11.1} follows from \eqref{eq:angle_cut}.
\end{proof}	

\begin{lemma}
\label{le:norm_bound}
If $0\leq u<w$ then for all $t\in\R$,
\begin{align}
\label{eq:norm_bound}
|\bx_u(t)|&\leq |\bx_w(t)|,\\
|\bx_u(t)|&\leq |\bx(t)|.\label{d12.3}
\end{align}
\end{lemma}
\begin{proof}
 If there are no collisions in $(u,w]$ then $\bx_u(t) = \bx_w(t)$ for all $t\in\R$. Assume that there is exactly one collision  in $(u,w]$, say, at time $t_1$. For $t\leq t_1$, we have  $\bx_u(t)=\bx_w(t)$. For $t>t_1$, we have
\begin{align*}
\bx_w(t) &= \bx_{t_1}(t) =\bx(t_1) + (t-t_1)\bv(t_1+), \\
\bx_u(t) & =\bx(t_1) + (t-t_1)\bv(t_1-),
 \end{align*}
 from which it follows that
 \begin{align*}
    |\bx_w(t)|^2 - |\bx_u(t)|^2 &= 2(t-t_1)\bx(t_1)\cdot( \bv(t_1+) - \bv(t_1-)).
 \end{align*}
Now \eqref{eq:norm_bound} follows  from \eqref{oc3.5}.

Let $u<t_1<t_{2}<\dots<t_{m+1}\leq w$ be all  collision times in $(u,w]$. Then, for all $t\in\R$,
\begin{align*}
|\bx_u(t)| \leq |\bx_{t_{1}}(t)| \leq |\bx_{t_{2}}(t)| \leq \cdots  \leq |\bx_{t_{m+1}}(t)| \leq |\bx_{w}(t)|.
\end{align*}

If we take $w$ greater than the time of the last collision then  $\bx_{w}\equiv\bx$, and  \eqref{d12.3} follows from \eqref{eq:norm_bound}.
\end{proof}

\begin{lemma}
\label{le:angle_x}
If $0\leq u<w$ then for all $s,t\geq u$, 
\begin{align}
\label{eq:angle_x}
\angle\left( \bx_w(s),\bx_w(t) \right) &\leq \angle\left( \bx_u(s),\bx_u(t) \right).
\end{align}
The inequality also holds when we replace $\bx_w$ by $\bx$.
\end{lemma}
\begin{proof}
Without loss of generality, assume that $u\leq s<t$. First, we prove \eqref{eq:angle_x} when there are no collisions in $(s,t)$.  We proceed by induction on the number of collisions in $(u,w]$. If there are no collisions in $(u,w]$, the inequality holds because $\bx_u=\bx_w$. 

Assume that the result holds when there are $m$  collisions in $(u,w]$. If there are $m+1$ collisions in $(u,w]$, say, at times  $u< t_1<t_{2}<\cdots<t_{m+1}\leq w$, then we have only $m$ collisions in $(t_1,w]$. Therefore, for $s,t \geq t_1$,
\begin{align}
\label{eq:02}
\angle(\bx_w(s),\bx_w(t)) &\leq \angle(\bx_{t_1}(s),\bx_{t_1}(t)).
\end{align}
We have assumed that there are no collisions in $(s,t)$. Hence, it will suffice to consider  the following two cases: (i) $u \leq s,t \leq t_1$, and (ii) $s,t \geq t_1$. 

If $u\leq s,t\leq t_1$  then $\bx_{w}(s)=\bx_{u}(s)$ and $\bx_{w}(t)=\bx_{u}(t)$, from which \eqref{eq:angle_x} follows.

Next consider the case when $s,t\geq t_1$. By the law of sines for the triangle with vertices $0, \bx_{t_1}(s)$ and $ \bx_{t_1}(t)$,
\begin{align*}
\frac{\sin\angle ( \bx_{t_1}(s), \bx_{t_1}(t) )}{t-s} &= \frac{\sin\angle(\bx_{t_1}(t),\bv_{t_1}(t_1+))}{|\bx_{t_1}(s)|}
= \frac{\sin\angle(\bx_{t_1}(t),\bv_{t_1}(t+))}{|\bx_{t_1}(s)|}.
\end{align*}
By the law of sines for the triangle with vertices  $0, \bx_{u}(s)$ and $\bx_{u}(t)$,
\begin{align*}
\frac{\sin\angle ( \bx_{u}(s), \bx_{u}(t) )}{t-s} &= \frac{\sin\angle(\bx_{u}(t),\bv_u(u+))}{|\bx_{u}(s)|}
= \frac{\sin\angle(\bx_{u}(t),\bv_{u}(t+))}{|\bx_{u}(s)|}.
\end{align*}
Therefore,
\begin{align}\label{d12.4}
\frac{\sin\angle ( \bx_{t_1}(s), \bx_{t_1}(t) )}{ \sin\angle ( \bx_{u}(s), \bx_{u}(t) ) } 
= \frac{\sin\angle(\bx_{t_1}(t),\bv_{t_1}(t+))}
{\sin\angle(\bx_{u}(t),\bv_u(t+))}\cdot \frac{|\bx_u(s)|}{|\bx_{t_1}(s)|} .
\end{align}

Our choice of $t_0=0$ in Remark \ref{re:zero} shows that  $\angle(\bx(0) , \bv(0+)) \leq \pi/2$, so, by Lemmas \ref{le:angle_dwn} and \ref{le:angle_cut},
\begin{align*}
\angle(\bx_{t_1}(t),\bv_{t_1}(t+))
\leq \angle(\bx_{0}(t),\bv_{0}(t+))
\leq \angle(\bx_0(0),\bv_0(0+))
= \angle(\bx(0) , \bv(0+)) \leq \pi/2.
\end{align*}
We can prove in the same way that $\angle(\bx_{u}(t),\bv_u(t+))\leq \pi/2$.
The sine function is increasing on $[0,\pi/2]$, so the estimates on the  angles that we have just obtained and  Lemma \ref{le:angle_cut} show that
\begin{align*}
\frac{\sin\angle(\bx_{t_1}(t),\bv_{t_1}(t+))}
{\sin\angle(\bx_{u}(t),\bv_u(t+))}\leq 1 .
\end{align*}
This and the bound $|\bx_u(s)|/|\bx_{t_1}(s)|\leq 1$, derived from Lemma \ref{le:norm_bound}, show that the right hand side of \eqref{d12.4} is bounded above by 1. Hence the left hand side is bounded by 1 as well, 
i.e.,
\begin{align*}%\label{d12.4}
\sin\angle ( \bx_{t_1}(s), \bx_{t_1}(t) )\leq  \sin\angle ( \bx_{u}(s), \bx_{u}(t) )  .
\end{align*}
When combined with \eqref{eq:02}, this yields
\begin{align*}%\label{d12.4}
\sin\angle ( \bx_{w}(s), \bx_{w}(t) )\leq  \sin\angle ( \bx_{u}(s), \bx_{u}(t) )  ,
\end{align*}
and this completes the proof of  \eqref{eq:angle_x} in the case when there are no collisions in $(s,t)$.

Next we prove the result for any number of collisions in $(s,t)$. Let $\left\{ t_1,\ldots,t_m\right\}$ be all  collision times in $(s,t)$. We have, by the previous case,
\begin{align*}
\angle(\bx_w(s),\bx_w(t)) 
&\leq \angle(\bx_w(s),\bx_w(t_1)) +\sum_{i=1}^{m-1} \angle(\bx_w(t_i),\bx_w(t_{i+1})) + \angle(\bx_w(t_m),\bx_w(t)) \\
&\leq  \angle(\bx_u(s),\bx_u(t_1)) +\sum_{i=1}^{m-1} \angle(\bx_u(t_i),\bx_u(t_{i+1})) + \angle(\bx_u(t_m),\bx_u(t)) \\
&= \angle(\bx_u(s),\bx_u(t)), 
\end{align*}
where the last equality holds because the velocity of $\bx_u$ is constant in $(s,t)$.

Finallly, if $w$ is larger than the last collision time, then  $\bx_w\equiv\bx$, so the result is also valid for $\bx$.
\end{proof}

\begin{lemma}
\label{le:norm_increasing}
The function $t\mapsto |\bx(t)|$ is increasing on $[0,\infty)$, and decreasing on $(-\infty,0]$. In particular, $|\bx(0)|\leq |\bx(t)|$ for all $t\in\R$.
\end{lemma}
\begin{proof}
The function $t\mapsto |\bx(t)|^2$ is differentiable  at all but finitely many times. It follows that for $t>s\geq 0$,
\begin{align*}	
|\bx(t)|^2 &=  |\bx(s)|^2 + \int_s^t 2\bx(u)\cdot\bv(u+)du \geq  |\bx(s)|^2,
\end{align*}
because $\bx(u)\cdot\bv(u+)\geq 0$ for $u\geq 0$, by Remark \ref{re:zero}. This shows that $t\mapsto |\bx(t)|$ is increasing in $[0,\infty)$. The second claim of the Lemma follows by time reversal, and the third claim follows easily from the first two.
\end{proof}

\section{Number of collisions}
\label{se:number}

The following is our key technical estimate.

\begin{theorem}\label{s28.1}
 The family of $n$ balls can be partitioned into two non-empty subfamilies such that no ball from the first family collides with a ball in the second family after time $100 n^3|\bx(0)|$.
\end{theorem}

\begin{proof}
Let $T = 18 \sqrt{n}( n-1)|\bx(0)| $. The speed of the moving point $\bx_0(t)$ in $\R^{dn}$ is equal to 1 because we assumed that $|\bv(0+)|^2 =1$. This and an elementary application of the law of sines, left to the reader, shows that  for any $t \geq T$,
\begin{align}\label{s27.3}
 \angle(\bx_0 (t),\bv(0+)) 
 \leq 2\frac{|\bx(0)|}{t}
\le \frac{1}{9\sqrt{n}(n-1)}.
\end{align}
This and \eqref{d11.1} imply that,
\begin{align*}
\angle(\bx (T),\bv(T+))
\leq \angle(\bx_0 (T),\bv_0(T+))
= \angle(\bx_0 (T),\bv(0+))\le \frac{1}{9\sqrt{n}(n-1)}
\leq \pi /2.
\end{align*}
Hence, in view of Remark \ref{re:zero}, we can apply 
Lemma \ref{le:angle_x}
to the trajectory $\{\bx(t), t\geq T\}$. Thus, for any $s,t\geq T$,
$\angle(\bx(s), \bx(t))\leq \angle(\bx_T(s), \bx_T(t)) $. This and
a simple analysis of the triangle with vertices $0, \bx_T(s)$ and  $\bx_T(t)) $ show that, for $t > s \geq T$,
\begin{align*}
\angle(\bx(s), \bx(t))
&\leq
\angle(\bx_T(s), \bx_T(t))
= \angle(\bx_T(s)-0, \bx_T(t)-0)\\
&\leq \angle(\bx_T(s)-0, \bx_T(t)-\bx_T(s))
= \angle(\bx_T(s), \bv_T(s+)).
\end{align*}
We combine this with Lemma \ref{le:angle_cut} and \eqref{s27.3} to obtain for $t>s\geq T$,
\begin{align}
\label{eq:01}
\angle(\bx(s), \bx(t)) 
&\leq
\angle(\bx_T(s), \bv_T(s+))
\leq \angle(\bx_0(s), \bv_0(s+)) \\
&=\angle(\bx_0(s), \bv_0(0+)) 
\leq \frac{1}{9\sqrt{n}(n-1)}.\nonumber
\end{align}
This, Lemma \ref{le:angle_cut}, and \eqref{s27.3}  imply that
for any $t>s \geq T$,
\begin{align*}
&\angle(\bv(s+), \bv(t+)) 
\leq
\angle(\bv(s+), \bx(s)) + \angle(\bx(s), \bx(t)) + \angle(\bx(t), \bv(t+))\\
&\ \leq \angle(\bx_0(s), \bv(0+)) + \angle(\bx(s), \bx(t)) + \angle( \bx_0(t), \bv(0+)) \leq \frac{1}{3\sqrt{n}(n-1)}.
\end{align*}
Hence, for all  $s,t \geq T$  and $k=1,\dots,n$,
\begin{align}\label{s27.4}
|v^k(s+) - v^k(t+)| \leq 
|\bv(s+)- \bv(t+)| \leq \angle(\bv(s+), \bv(t+))
 \leq \frac{1}{3\sqrt{n}(n-1)}.
\end{align}

Since $|\bv(T+)|^2 =1$, there exists $k_1$ such that $|v^{k_1}(T+)| \geq 1/\sqrt{n}$. The total momentum of the system is zero, so 
\begin{align}
\label{eq:velo}
\frac1{\sqrt{n}}\leq \max_{1\leq k\leq n} |v^k(T+)| \leq \max_{1\leq i,j\leq n} |v^i(T+)-v^j(T+)|.
\end{align}
Let $i^*$ and $j^*$ realize the maximum on the right hand side of \eqref{eq:velo}. We will argue that we can partition $\{1,2,\dots,n\}$ into two non-empty families $N_1$ and $N_2$ such that 
\begin{align}\label{s27.8}
|v^{\ell}(T+) - v^{k}(T+)|\geq \frac1{\sqrt{n}(n-1)}
\end{align}
 for all $\ell\in N_1$ and $k\in N_2$. This is possible because otherwise there would be a sequence $j_1 , j_2, \dots, j_m $ such that $m \leq n$, $j_1=i^*$, $j_m=j^*$ satisfying $|v^{j_{i-1}}(T+) - v^{j_i}(T+)|< 1/((n-1)\sqrt{n})$
for all $i$, which, by the triangle inequality, would contradict \eqref{eq:velo}.

Next we will show that 
two balls, one with index in $N_1$ and the other one with index in $N_2$, will not collide after time $T_*:=T+3CT \sqrt{n}(n-1) $, where $C$ is to be chosen later. Consider  $\ell\in N_1$, $k\in N_2$, and $t\geq T_*$. Writing $x^i(t) = x^i(T) + v^i(T+)(t-T) + \int_T^t (v^i(s+)-v^i(T+))ds$, for $i=k,\ell$, we obtain
\begin{align*}
|x^k(t) - x^\ell(t)| 
&\geq \left| v^k(T+)-v^\ell(T+)\right|(t-T)  - \left| \int_T^{t} (v^k(s+) - v^k(T+) ) ds\right|\\
&\qquad\qquad - \left|\int_T^{t} (v^\ell(s+) - v^\ell(T+) ) ds \right| - | x^k(T) - x^\ell(T) |.
\end{align*}
Using estimates \eqref{s27.8} and \eqref{s27.4},  we obtain
\begin{align*}
|x^k(t) - x^\ell(t)|  
&\geq  \frac{t-T}{3\sqrt{n}(n-1)} -  | x^k(T) - x^\ell(T) | \geq CT -  | x^k(T) - x^\ell(T) |.
\end{align*}
Since $|\bv(s+)|=1$ for all $s$,
\begin{align*}
| x^k(T) - x^\ell(T) |\leq\sqrt{2}|\bx(T)|\leq \sqrt{2}\left( |\bx(0)|+\int_0^T|\bv(u+)|du\right) =\sqrt{2}(|\bx(0)|+T).
\end{align*}
 We have  $T = 18 \sqrt{n}( n-1)|\bx(0)| \geq 18 \sqrt{2}( 2-1)|\bx(0)|
= 18 \sqrt{2}|\bx(0)|$, so
\begin{align*}
|x^k(t) - x^\ell(t)|  
&\geq (C-\sqrt2)T - \sqrt{2}|\bx(0)| \geq \left( 18\sqrt{2}(C-\sqrt 2)  -\sqrt 2 \right) |\bx(0)| .
\end{align*}
By choosing $C=1.61$ we ensure that the right hand side above is  greater than $(5/2)\sqrt{2} |\bx(0)|$.

Since all balls have radii 1 and cannot overlap, we must have $|x^1(0)| + |x^2(0)| \geq 1$. Hence,  $|\bx(0)| \geq (|x^1(0)|^2 + |x^2(0)|^2)^{1/2} \geq \sqrt{2}/2$. It follows that for $t\geq T_*$,
\begin{align}\label{a7.1}
|x^k(t) - x^\ell(t)|  > (5/2)\sqrt{2} |\bx(0)| \geq (5/2)\sqrt{2} \cdot \sqrt{2}/2 = 5/2,
\end{align}
so balls $k$ and $\ell$ do not collide after time $T_*$.  It is easy to check that $T_*<100 n^3 |\bx(0)|$, assuming that $C=1.61$.
\end{proof}

We will need the following quantitative version of Theorem \ref{s28.1}.

\begin{corollary}\label{a7.2}
 The family of $n$ balls can be partitioned into two non-empty subfamilies
 $ N_1$ and $ N_2$
such that if  $\ell\in N_1$, $k\in N_2$, and $t\geq 100 n^3 |\bx(0)|$ then
$|x^k(t) - x^\ell(t)|  >  5/2$.
\end{corollary}

\begin{proof}
The claim has been proved in \eqref{a7.1}.
\end{proof}

\begin{corollary}\label{d9.1}
Let $T = \inf\{t\geq 0: |\bx(t)| \geq (1+100 n^3)|\bx(0)|\}$.
 The family of $n$ balls can be partitioned into two non-empty subfamilies such that no ball from the first family collides with a ball in the second family after time $T$.
\end{corollary}

\begin{proof}
The speed  $|\bv(t)|$ of the trajectory $\bx(t)$ in $\R^{nd}$ is equal to 1 because we assumed that $|\bv(0)|^2 =1$ and, therefore, by the conservation of energy, $|\bv(t)|^2 =1$ for all $t\geq 0$.
Since $|\bx(T)-\bx(0)| \geq 100 n^3 |\bx(0)|$, we conclude that $T \geq 100 n^3 |\bx(0)|$. The corollary follows from this and Theorem \ref{s28.1}.
\end{proof}

\begin{theorem}
\label{th:nf}
Suppose that 
 at the time of collision of any two balls, the distance between any other pair of balls is greater than $\delta \in (0,1)$.
Then the total number of collisions is bounded by $73^n(n!)^{3/2} (\log(5n))^n \delta^{-1}$, for all $n\geq 3$.
\end{theorem}

\begin{proof}
Fix $\delta\in(0,1)$, and let $K_\delta(n)$ be the maximum possible number of collisions among $n$ balls
if the assumption of the theorem is satisfied. We will use induction to estimate $K_\delta(n)$.
Obviously, $K_\delta(2)=1$. 

By Theorem \ref{s28.1},
the family of $n$ balls can be partitioned into two non-empty subfamilies such that no ball from the first family collides with a ball in the second family after a finite time given by $T=100n^3|\bx(0)|$. The total number of collisions within each subfamily is bounded by $K_\delta(n-1)$. Hence, the total number of collisions for the whole family of $n$ balls is bounded by $2K_\delta(n-1)$ on the interval $[T,\infty)$.

To estimate the number of collisions in $[0,T)$, we will split the argument into two parts, depending on the value of $|\bx(0)|$. Let $R= 6n^{3/2}$. 

\bigskip
\emph{Step 1}.
In this step of the proof, we will assume that
$|\bx(0)| \leq R$. Suppose that  balls $B_j$ and $B_k$ collide at a time $t\in [0,T)$. These balls will not collide again until the trajectory of at least one of them is changed. Hence, there will be no collisions after $t$ until some other pair of balls (one of which can be $B_j$ or $B_k$) collide. Since the distance between any pair of balls, except $(B_j,B_k)$, is bounded below by $\delta$ at time $t$ and all speeds are bounded by 1, there will be no collisions in the time interval $(t, t+ \delta/2)$. We let $J_m =[(m-1)\delta/2,m\delta/2)$ for $m\geq 1$. The argument given above shows that there is at most one collision in $J_m$, and thus the number of collisions in $[0,T)$ is at most
\begin{align}
\label{m16.1}
\frac T {\delta/2}+1
&=\frac{100n^3|\bx(0)|}{\delta/2}+1 = 200n^3|\bx(0)|\delta^{-1}+1
\leq 200n^3 R \delta^{-1}+1\\
&\leq 200n^3 6n^{3/2} \delta^{-1}+1 \leq 1201n^{9/2}\delta^{-1}
\leq 1224n^{9/2}\delta^{-1}.\notag
\end{align}
The reader may be puzzled by the last inequality. Later, in \eqref{m16.2}, we will factor 1224 into the product of 72 and 34, to have a more elegant formula.

\bigskip
\emph{Step 2}. Next suppose that $|\bx(0)| > R$. By Lemma \ref{le:norm_increasing}, we have $|\bx(t)|>R $ for all $t\geq 0$.
Let $D(t) = \max_{1\leq i,j\leq n} |x^i(t) - x^j(t)|$.
We have assumed in (A3) that the center of mass is at the origin. Hence,  $|x^k(t)| \leq D(t)$ for all $k$ and $t$. There must exist $k$ such that $|x^k(t) | \geq D(t)/2$. It follows that for all $t$,
\begin{align}\label{s30.1}
D(t)/2 \leq |\bx(t)| \leq \sqrt{n} D(t).
\end{align}

Consider any $s\in [0,T)$. 
There exist $k_1$ and $k_2$ such that 
\begin{align*}
|x^{k_1}(s) - x^{k_2}(s)| = D(s) \geq |\bx(s)| n^{-1/2} \geq |\bx_0(s)| n^{-1/2},
\end{align*}
where in the last inequality we have used \eqref{d12.3}. 

We will argue that one can partition $\{1,2,\dots,n\}$ into two non-empty families $N_1(s)$ and $N_2(s)$ such that  
\begin{align}\label{s28.2}
|x^{j}(s) - x^{k}(s)|\geq  |\bx_0(s)| n^{-3/2}
\end{align}
 for all $j\in N_1(s)$ and $k\in N_2(s)$. This is possible because otherwise there would be a sequence $j_1 , j_2, \dots, j_m $ such that $m \leq n$, $j_1=k_1$, $j_m=k_2$ and $|x^{j_{i-1}}(s) - x^{j_i}(s)|<  |\bx_0(s)| n^{-3/2}$
for all $i=2,\dots,m$, and, by the triangle inequality, this would contradict the fact that $|x^{k_1}(s) - x^{k_2}(s)| \geq  |\bx_0(s)| n^{-1/2}$.

For any $s\geq 0$ we have
\begin{align}\label{a4.1}
|\bx_0(s)|^2 = |\bx(0)|^2 + 2s \bv(0+)\cdot\bx(0) + s^2 \geq |\bx(0)|^2,
\end{align}
because we chose $t_0=0$ in Remark \ref{re:zero} and, therefore,  $\bv(0+)\cdot\bx(0) \ge 0$.

The maximum speed of any ball is 1 so \eqref{s28.2} and \eqref{a4.1} imply that  for all $j\in N_1(s)$, $k\in N_2(s)$ and $t \in [s, s+ |\bx_0(s)| n^{-3/2}/4]$, 
\begin{align}\label{a7.3}
|x^{j}(t) - x^{k}(t)|&\geq 
|\bx_0(s)| n^{-3/2} - 2 |\bx_0(s)| n^{-3/2}/4 = |\bx_0(s)| n^{-3/2}/2 \\
&\geq |\bx(0)| n^{-3/2}/2 \geq R n^{-3/2}/2 \geq
6 n^{3/2} n^{-3/2}/2=3 >2.\notag
\end{align}
This implies that there are no collisions of balls belonging to  different subfamilies $ N_1(s)$ and $ N_2(s)$ in the interval $[s, s +  |\bx_0(s)|  n^{-3/2}/4]$.
 The total number of collisions within each subfamily is bounded by $K(n-1)$. Hence, the total number of collisions for the whole family of $n$ balls is bounded by $2K(n-1)$ on the interval $[s, s+ |\bx_0(s)|  n^{-3/2}/4]$.

Let $s_0 = 0$ and $s_{j+1} = s_{j} +  |\bx_0(s_j)|  n^{-3/2}/4$ for $j\geq 0$. It follows from \eqref{a4.1} that $|\bx_0(s)|\geq \sqrt{|\bx(0)|^2+s^2}\geq(|\bx(0)|+s)/\sqrt{2}$. It is straightforward to show by induction that 
\begin{align*}
s_j \geq |\bx(0)| \left( \left( 1+ \frac{1}{4\sqrt{2}n^{3/2}} \right)^j -1\right).
\end{align*}

Let $j_* = \min\{j: s_j \geq T\}$. The argument in the previous paragraph shows that in each of the $j_*$ intervals $[s_j,s_{j+1})$ that cover $[0,T)$ there are at most $2K_\delta(n-1)$ collisions. We obtain that there are no more than $2j_*K_\delta(n-1)$ collisions in $[0,T)$.

We will show that 
\begin{align}\label{a7.5}
j_*\leq 18n^{3/2}\log(5n).
\end{align}
The definition of $j_*$  implies that
\begin{align*}%\label{a8.6}
 \log(1+100n^3) &\geq (j_*-1)\log \left(  1+ \frac{1}{4\sqrt{2}n^{3/2}} \right) \geq (j_*-1) \left(  1+ {4\sqrt{2}n^{3/2}} \right)^{-1},
\end{align*}
where we used the well known inequality $\log(1+x^{-1})\geq (1+x)^{-1}$, for $x>0$. 
This implies that
\begin{align}\label{a8.10}
j_* \leq 1+(1+4\sqrt{2}n^{3/2}) \log(1+100n^3).
\end{align}
For $n\geq 3$,
\begin{align}\label{a8.7}
&18 n^{3/2} \log(5n) \geq 18\cdot 3^{3/2} \log(15) > 250,\\
& (1 +  4\sqrt{2}n^{3/2} ) / (4\sqrt{2}n^{3/2}) = 1 + 1/ (4\sqrt{2}n^{3/2}) \leq  1 + 1/ (4\sqrt{2} \cdot 3^{3/2}) < 1.04,\label{a8.8}\\
&\log(1+100n^3) = 3\log(5n) + \log(0.8 + 1/(125 n^3))\label{a8.9}\\
&\quad \leq 3\log(5n) + \log(4/5 + 1/(125 \cdot 3^3))  < 3\log(5n). \notag
\end{align}
We divide \eqref{a8.10}  by $18 n^{3/2} \log(5n) $ and use \eqref{a8.7}-\eqref{a8.9}  to get
\begin{align*}
\frac{j_*}{18 n^{3/2} \log(5n)} &\leq \frac{1}{18 n^{3/2} \log(5n)} + \frac{4\sqrt{2}}{6} \cdot  \frac{1+4\sqrt{2}n^{3/2}}{ 4\sqrt{2} n^{3/2} } \cdot \frac{ \log(1+100n^3) }{3 \log(5n)} \\
&< 0.004 + 0.943 \cdot 1.04 \cdot 1 <1.
\end{align*}
This completes the proof of \eqref{a7.5}.

 We summarize this step by stating that, when $|\bx(0)|>R$, an upper bound for the number of collisions in $[0,T)$ is 
\begin{align}
\label{m22.1}
36n^{3/2}\log(5n) K_\delta(n-1).
\end{align}

\bigskip
\emph{Step 3}. Estimates \eqref{m16.1} and \eqref{m22.1}, and the argument given  at the beginning of this proof, give us a bound for the number of collisions in $[0,\infty)$. We apply time reversal to double such bound and obtain a bound for the  number of collisions on $(-\infty, \infty)$. Therefore, we have that  $K_\delta(2)=1$, and
\begin{align}
\label{m16.2}
K_\delta(n) &\leq 4K_\delta(n-1) + 72n^{3/2}\max\left( 34n^{3}\delta^{-1} , \log(5n)K_\delta(n-1)\right),
\end{align}
for any $n\geq 3$. 

Let $\varphi_\delta(n) = 73^n(n!)^{3/2} (\log(5n))^n\delta^{-1}$. We will prove by induction that $K_\delta(n)\leq \varphi_\delta(n)$ for all $n\geq 2$. For $n=2$, this is direct since $K_\delta(2)=1$, and $\delta<1$. Assume that $K_\delta(n-1)\leq \varphi_\delta(n-1)$ for some $n\geq3$. We have 
\begin{align*}
34\leq \log(15)\cdot 2^{3/2}(\log 10)^2 \leq \log(5n) ((n-1)!)^{3/2}(\log(5(n-1)))^{n-1}.
\end{align*}
Note that $n^3\leq 73^{n-1}$ because $\log(n)/(n-1) \leq 1< \log(73)/3$. From these inequalities, it follows that
\begin{align*}
34n^3\delta^{-1} \leq \log(5n)\varphi_\delta(n-1).
\end{align*}
We use this inequality, together with \eqref{m16.2}, and the induction hypothesis to obtain 
\begin{align}\label{a8.11}
K_\delta(n)&\leq 4\varphi_\delta(n-1) + 72n^{3/2}\log(5n)\varphi_\delta(n-1) \leq 73n^{3/2}\log(5n)\varphi_\delta(n-1) \\
&= 
73^n(n!)^{3/2} (\log(5n))^n\delta^{-1} = \varphi_\delta(n).\notag
\end{align}
This completes the proof.
\end{proof}

\begin{proof}[Proof of Theorem \ref{th:nc}]

Recall the definition of $K_\delta(n)$ from the proof of Theorem \ref{th:nf}. The bound $n!\leq e n^{n+1/2}e^{-n}\leq n^{n+1/2}$ is related to Stirling's approximation. We use this bound and \eqref{a8.11} to obtain
\begin{align}\label{a8.12}
\log K_\delta(n) &\leq n\log 73  +\frac32 \log n! + n \log\log(5n) + \log \delta^{-1}  \\
&\leq \left(\frac{\log 73}{\log n}  + \frac32\frac{n+\frac12}{n} + \frac{ \log\log(5n)}{\log n} \right) n\log n  + \log\delta^{-1}\notag\\
&= \left( \frac32 + \frac{\log 73}{\log n}   + \frac{3}{4n} + \frac{ \log\log(5n)}{\log n} \right) n\log n  + \log\delta^{-1} .\notag
\end{align}
We will next find a bound for the last three terms in the parenthesis in the last formula. We first use the fact that for $x\geq 1$, $\log(x)\leq \sqrt{x}$, to obtain
\begin{align*}
\sqrt{\log n} \left( \frac{\log 73}{\log n}  + \frac{3}{4n} + \frac{ \log\log(5n)}{\log n} \right) &\leq \frac{\log{73}}{\sqrt{\log n}} + \frac{3}{4n^{3/4}} + \sqrt{\frac{\log(5n)}{\log n}} \\
&= \frac{\log{73}}{\sqrt{\log n}} + \frac{3}{4n^{3/4}} + \sqrt{1 +\frac{\log 5}{\log n }}.
\end{align*}
By monotonicity, it follows that for $n\geq 3$,
\begin{align*}
\sqrt{\log n} \left( \frac{\log 73}{\log n}  + \frac{3}{4n} + \frac{ \log\log(5n)}{\log n} \right)
&\leq \frac{\log{73}}{\sqrt{\log 3}} + \frac{3}{4\cdot 3^{3/4}} + \sqrt{1 +\frac{\log(5)}{\log(3)}} < 6 .
\end{align*}
We combine this estimate with \eqref{a8.12} to obtain
\begin{align*}
\log K_\delta(n) &\leq \left( \frac32 + \frac{6}{\sqrt{\log n}} \right) n\log n + \log \delta^{-1}.
\end{align*}
It follows that for $n\geq \exp (36\eps^{-2})$, 
\begin{align*}
K_\delta(n) \leq \delta^{-1} n^{3n/2 + \eps n }.
 \end{align*}
By setting $\delta = n^{-n}$, we obtain the theorem.
\end{proof}

\section{Connected configurations}\label{connect}

We recall some notation and definitions
from Section \ref{se:intro}.
We assume that all balls have radii equal to 1. Consider $\rho>0$.
Let $\Gamma_\rho(t)$ be the graph whose vertices are balls $B_1, B_2, \dots, B_n$. Two vertices $B_j$ and $B_k$ are connected by an edge in $\Gamma_\rho(t)$  if and only if $|x^j(t)-x^k(t)|\leq 2+\rho$.

We will say that a subfamily $\{B_{i_1}, B_{i_2}, \dots, B_{i_k}\}$ of balls is $\rho$-connected in $[s,u]$ if for every $t \in [s,u]$, all balls $\{B_{i_1}, B_{i_2}, \dots, B_{i_k}\}$ belong to a  connected component of $\Gamma_\rho(t)$ (the connected component may depend on $t\in [s,u]$). 

\begin{proof}[Proof of Theorem \ref{th:Np}]
\emph{Step 1}.
We will eventually use the assumption that $\rho\leq n^{-n}$, but we start only assuming that $\rho$ is a fixed number in $(0,1/4)$. 
Consider some $s\geq 0$, and
 $1\leq i,j \leq n$, $i\ne j$. 
Let $\tau^{ij}_0(s)=s$, and for $k\geq 0$,
\begin{align*}
\tau^{ij}_{2k+1}(s)&= \inf\{u\geq \tau^{ij}_{2k}(s): |x^i(u) - x^j(u)| \leq 2+ \rho/2\}, \\
\tau^{ij}_{2k+2}(s)&= \inf\{u\geq \tau^{ij}_{2k+1}(s): |x^i(u) - x^j(u)| >2+ \rho\},
\end{align*}
with the convention that $\inf \emptyset = \infty$.
For $0 \leq s < t < \infty$, 
let $\sigma^{ij}(s,t)$ be the largest $k$ such that $\tau^{ij}_{2k}(s) \leq t$, 
i.e., $\sigma^{ij}(s,t)$ is the number of upcrossings of the interval $[2+\rho/2, 2+\rho]$ by the function $u\to |x^i(u) - x^j(u)|$ on the interval $[s,t]$.  Note that for fixed $s$, this defines a non-decreasing function in $t$, which implies that $\sigma^{ij}(s,\infty) := \sup_{t \in (s,\infty)} \sigma^{ij}(s,t) = \lim_{t\to\infty} \sigma^{ij}(s,t)$ is well defined. Since there are only finitely many  upcrossings, the supremum in this definition can be replaced with the maximum.

For $0 \leq s < t \leq \infty$, let $S(s,t) = \sum_{i<j} \sigma^{ij}(s,t)$.
Let $M_\rho(n)$ be the supremum of $S(0,\infty)$
over all initial conditions for the family of $n$ balls.
We will use induction to estimate $M_\rho(n)$. Obviously, $M_\rho(2)=1$.

By Corollary \ref{a7.2},
 the family of $n$ balls can be partitioned into two non-empty subfamilies
 $ N_1$ and $ N_2$
such that if  $\ell\in N_1$, $k\in N_2$, and $t\geq T:=100 n^3 |\bx(0)|$ then
$|x^k(t) - x^\ell(t)|  >  5/2$.
We have assumed that $\rho<1/4$ so if  $\ell\in N_1$ and $k\in N_2$ then $\sigma^{\ell k}(T,\infty)=0$. Hence,
\begin{align}\label{a8.1}
S(T,\infty) =  \sum_{i<j;\ i,j\in N_1} \sigma^{ij}(T,\infty)
+\sum_{i<j;\ i,j\in N_2} \sigma^{ij}(T,\infty),
\end{align}
and, therefore,
\begin{align}\label{a7.4}
M_\rho(n) \leq \sup S(0,T) + 2 M_\rho(n-1)+n^2,
\end{align}
where the supremum is taken over all initial conditions. We added $n^2$ to  the right hand side of \eqref{a7.4} to account for the possibility that 
$\tau^{ij}_{2k+1}(0)< T <
\tau^{ij}_{2k+2}(0)$ for some $i,j,k$ (the number of pairs $(i,j)$ is bounded by $n^2$).

\bigskip
\emph{Step 2}.
To estimate $S(0,T)$, we will split the argument into two parts, depending on the value of $|\bx(0)|$. Let $R= 6n^{3/2}$. 

First assume that
$|\bx(0)| \leq R$. 
The distance between any two balls cannot change at a rate faster than 1 because we have assumed that $|\bv(t)|=1$ for all $t$. Hence, $\tau^{ij}_{2(k+1)}(s)-\tau^{ij}_{2k}(s) \geq \rho/2$ for all $k\geq 1$ such that $\tau^{ij}_{2k}(s)$ is finite. Also, $\tau^{ij}_{2}(s)-\tau^{ij}_{0}(s) \geq \rho/4$ because it is possible that $\tau^{ij}_1(s)=s$.

It follows  that $\tau^{ij}_{2k}(0)  \geq (2k-1)\rho/4$. Thus,
\begin{align*}
\sigma^{ij} (0,T) &\leq \frac{2 T }{\rho} +\frac12
= \frac {200n^3|\bx(0)|} {\rho} +\frac12 \\
&\leq \frac {200n^3 R} {\rho}+\frac1\rho
= \frac {200n^3 6n^{3/2}} {\rho}+\frac1\rho
\leq1201 n^{9/2}/ \rho.
\end{align*}
It follows that
\begin{align}\label{a7.6}
S(0,T) \leq \sum_{i<j; \ 1\leq i,j \leq n} 1201 n^{9/2}/ \rho
= \binom n 2 1201 n^{9/2}/ \rho \leq 601 n^{13/2}/ \rho.
\end{align}

Next suppose that $|\bx(0)| > R$. 
We will use the notation and estimates proved in Step 2 of the proof of Theorem \ref{th:nf}. 

Let $s_0 = 0$ and $s_{j+1} = s_{j} +  |\bx_0(s_j)|  n^{-3/2}/4$ for $j\geq 0$. 
It follows from \eqref{a7.3} that for each $j \geq 0$,
 the family of $n$ balls can be partitioned into two non-empty subfamilies
 $ N^j_1$ and $ N^j_2$
such that if  $\ell\in N^j_1$, $k\in N^j_2$, and $t\in  [s_j, s_j+ |\bx_0(s_j)| n^{-3/2}/4]$ then
$|x^k(t) - x^\ell(t)|  \geq 3$.
A reasoning based on a formula similar to \eqref{a8.1} yields $S(s_j,s_{j+1}) \leq 2 M_\rho(n-1)$ for all $j$. 
Let $j_* = \min\{j: s_j \geq T\}$.  In the following formula, we add $n^2$ to the estimate for the same reason as in \eqref{a7.4}. 
We use \eqref{a7.5} to get,
\begin{align}\label{a7.7}
S(0,T) &\leq \sum_{j=0}^{j_*}  \left( S(s_j, s_{j+1}) + n^2 \right)
\leq j_* ( 2 M_\rho(n-1) + n^2)\\
&\leq 18n^{3/2}\log(5n) ( 2 M_\rho(n-1) + n^2).\notag
\end{align}

\bigskip
\emph{Step 3}. We combine \eqref{a7.4}, \eqref{a7.6} and \eqref{a7.7} to obtain
\begin{align}\label{a8.2}
M_\rho(n) &\leq  2 M_\rho(n-1)+n^2
+ \max(601 n^{13/2}/ \rho, 18n^{3/2}\log(5n) ( 2 M_\rho(n-1) + n^2))%\\
%&= 2 M_\rho(n-1)
%+ \max\left(601 n^{13/2}/ \rho, 40n^{7/2}\log(5n)  M_\rho(n-1)\right )\notag\\
%&= 2 M_\rho(n-1)
%+ 40n^{7/2} \max\left(15 n^{3}/ \rho, \log(5n)  M_\rho(n-1)\right ),\notag
\end{align}
for all $n\geq 3$. 

We will prove by induction that if we set $\varphi_\rho(n) =  38^n (n!)^{3/2} (\log(5n))^n\rho^{-1}$ for $n\geq 2$, then
%$M_\rho(n)\leq \varphi_\rho(n)$ 
\begin{align}\label{a10.1}
M_\rho(n)\leq   \varphi_\rho(n)  - \frac{(n+1)^2}{2},
\end{align}
%\begin{align}\label{a8.3}
%M_\rho(n)\leq \varphi_\rho(n)
%:= 82^n(n!)^{7/2} (\log(5n))^n\rho^{-1}
%\end{align}
for all $n\geq 2$. For $n=2$, the bound holds because $M_\rho(2)=1$ and $\rho<1$. Assume that $M_\rho(n-1)\leq \varphi_\rho(n-1) -n^2/2$ for some $n\geq3$. From \eqref{a8.2} and the induction hypothesis, we obtain
\begin{align}
\label{a10.2}
M_\rho(n) &\leq 2 \varphi_\rho(n-1) +\max\left( 601n^{13/2}/\rho,36 n^{3/2}\log(5n) \varphi_\rho(n-1) \right).
\end{align}
By the Stirling-type bound $n!\geq\sqrt{2\pi n}(n/e)^n$, we have
\begin{align}\label{a11.3}
36 n^{3/2}\log(5n) \varphi_\rho(n-1)\rho & \geq 36^n (n!)^{3/2} (\log(5(n-1)))^n \\
&\geq 36^n (\sqrt{2\pi n} (n/e)^n)^{3/2} (\log(5(n-1)))^n \notag\\
&=  (36e^{-3/2})^n (2\pi)^{3/4} n^{3n/2+3/4} (\log(5(n-1)))^n.\notag
\end{align}
We claim that the last expression is larger than $601 n^{13/2}$ for all $n\geq 3$. This can be verified directly for $n=3$. For $n\geq 4$, note that $3n/2+3/4\geq 13/2$, so the right hand side of \eqref{a11.3} is larger than
\begin{align*}
(36e^{-3/2})^n (2\pi)^{3/4} n^{13/2} (\log(5(n-1)))^n \geq (36e^{-3/2})^4  (2\pi)^{3/4} n^{13/2} (\log(15))^4 \geq 601 n^{13/2}.
\end{align*}
Thus, \eqref{a10.2} becomes
\begin{align}
\label{a10.2a}
M_\rho(n) &\leq \left( 2  + 36 n^{3/2}\log(5n) \right) \varphi_\rho(n-1).
\end{align}

Before completing the proof, we will establish a simple inequality. For $n\geq 2$, we have $2n\geq 2^2\geq(1+2/n)^2$, from which the inequality $2n^3\geq (n+2)^2$ follows. 
By the change of variable, 
$(n-1)^3 \geq (n+1)^2/2$ for $n\geq 3$. Since $\rho<1$, it follows that 
\begin{align}\label{a11.4}
\varphi_\rho(n)\geq 38 (n!)^{3/2} \geq 38 (n-1)^3\geq 19 (n+1)^2.
\end{align}
 Going back to \eqref{a10.2a}, we compute
\begin{align*}
M_\rho(n) &\leq \left( 2  + 36 n^{3/2}\log(5n) \right) \varphi_\rho(n-1) 
= \frac{ 2  + 36 n^{3/2}\log(5n) }{ 38 n^{3/2} \log(5n) } \varphi_\rho(n)  \\
&\leq \frac{37}{38}  \varphi_\rho(n) =  \varphi_\rho(n) -  \frac{\varphi_\rho(n) }{38} \leq   \varphi_\rho(n) -  \frac{ (n+1)^2 }{2}, 
\end{align*}
where in the last inequality follows from \eqref{a11.4}. This completes the proof of \eqref{a10.1}.

It follows from \eqref{a10.1} that
\begin{align}
\label{a11.4a}
4n^4M_\rho(n) &\leq 
4n^4(   \varphi_\rho(n)  - (n+1)^2/2)
\leq 
4n^4   \varphi_\rho(n)  \\
&= 4n^4 38^n (n!)^{3/2} (\log(5n))^n\rho^{-1}  
\leq 4n^4 38^n (n^n)^{3/2} (\log(5n))^n\rho^{-1}\notag \\
&\leq
\rho^{-1} n^{(3/2+o(1))n}.\notag
\end{align}

%We have that
%\begin{align*}
%15\leq \log(15)\cdot 2^{7/2}(\log 10)^2 \leq \log(5n) ((n-1)!)^{7/2}(\log(5(n-1)))^{n-1}.
%\end{align*}
%We also have that $n^3\leq 82^{n-1}$, because $\log(n)/(n-1) \leq 1< \log(82)/3$. From these inequalities, it follows that
%\begin{align*}
%15n^3\rho^{-1} \leq \log(5n)\varphi_\rho(n-1).
%\end{align*}
%We use this inequality, together with \eqref{m16.2}, and the induction hypothesis to obtain 
%\begin{align*}
%M_\rho(n)&\leq 2\varphi_\rho(n-1) + 80n^{7/2}\log(5n)\varphi_\rho(n-1) \\
%&\leq 82n^{7/2}\log(5n)\varphi_\rho(n-1) \leq \varphi_\rho(n).
%\end{align*}

\bigskip
\emph{Step 4}.
Let $t_1 \leq t_2 \leq \dots\leq t_{M-1}$ be all finite times of the form $\tau^{ij}_{2k}(0)$,
for any $i,j$ and $k>1$. Let $t_0=0$ and $t_{M}=\infty$. Note that we have $M_\rho(n)\geq M-1\geq M/2$ for $n\geq 3$.
We have assumed in the theorem that the total number of collisions is not smaller than $N\geq n^{(3/2+\eps)n}\rho\geq n^{n/2}$. By time reversal, we do not loose any generality by assuming that there are at least $N/2$ collisions in $[0,\infty)$. It follows that  
there exists $j' \in \left\{0,\ldots,M-1 \right\}$ such that the interval $[t_{j'}, t_{j'+1})$  contains at least $N/(2M)$ collision times.
Let $u_i$, $i=1,\dots, i'$ be times such that  $t_{j'}\leq u_1 < u_2 < \dots u_{i'} < t_{j'+1} $ and $u_i$'s are the only times in $[t_{j'}, t_{j'+1})$ with the property that at time $u_i$ there is a collision of balls which did not collide in $[t_{j'}, u_i)$. There are $\binom n 2$ pairs of balls so $i' \leq n^2$. Therefore, one of the intervals $[u_i, u_{i+1})$ contains at least  $N/(2M n^2)$ collision times.
Fix an interval $[u_i, u_{i+1})$ with this property.
Let $J$ be the family of all pairs $(k_1, k_2)$ such that balls $B_{k_1}$ and $B_{k_2}$ collide in $[t_{j'},u_{i}]$. 
For $1\leq \ell,j \leq n$, we say that $\ell\sim j$ if there exist $k_1 = \ell, k_2,\dots,k_{m-1}, k_m=j$ such that $(k_r, k_{r+1})\in J$ for all $r$. This is an equivalence relation so it
 partitions $J$ into equivalence classes $J_1, J_2, \dots, J_{m'}$. 
Note that $m'\leq n^2$ so there exists $m_*$ such that there were at least 
$N/(2M n^4)$ collision times between balls $B_{k_1}$ and $B_{k_2}$ with $(k_1, k_2) \in J_{m_*}$ in $[u_i, u_{i+1})$.
For every $(k_1, k_2)\in J_{m_*}$, there was a time $t\in[t_{j'},u_{i+1}]$ such that $|x^{k_1} (t) - x^{k_2}(t)| = 2$. Since there are no times of the form $\tau^{ij}_{2k}(0)$ in $(t_{j'}, t_{j'+1})$, there are no such times in $(u_i, u_{i+1})\subset (t_{j'}, t_{j'+1})$. Hence $|x^{k_1} (t) - x^{k_2}(t)| \leq 2+\rho$ for all $t\in[u_i, u_{i+1})$ and $(k_1, k_2)\in J_{m_*}$.
Let $\calB :=\{B_{i_1}, B_{i_2}, \dots, B_{i_k}\}$ be the family of all balls  
such that $(i_\ell, i_r) \in J_{m_*}$ for all $\ell$ and $r$. We have proved that $\calB$  is $\rho$-connected on $[u_i, u_{i+1})$ and there are at least $N/(2M n^4)$ collisions  between these balls in this interval. 

%We now take $\rho=n^{-\frac 1 3 n^\alpha}$.
It follows from \eqref{a11.4a}, and the assumptions of the theorem that
\begin{align*}
N/(2M n^4) &\geq  N/(4M_\rho(n) n^4) \geq N\rho n^{-(3/2+o(1))n}.
%
%\frac{n^{n^\alpha}}{2\cdot 38^n (n!)^{3/2} (\log(5n))^n\rho^{-1} n^4}\\
%&= \frac{n^{n^\alpha}}{2\cdot38^n(n!)^{3/2} (\log(5n))^n n^{\frac 1 3 n^\alpha} n^4}.
\end{align*}
We have shown that $\calB$  is $\rho$-connected on  $[u_i, u_{i+1})$ and there are at least $N\rho n^{-(3/2+o(1))n}$ collisions  between these balls in this interval. 
\end{proof}

\begin{proof}[Proof of Corollary \ref{a11.2}] We apply Theorem \ref{th:Np} with $N =  n^{(5/2 + \eps)n}$ and $\rho=n^{-n}$. In this case, we have $\log(N\rho) = (3/2 + \eps)n\log n$, and thus,  we obtain a family $\calB$ of balls that is $\rho$-connected on some interval $[t_1,t_2]$, and the number of collisions among balls in $\calB$ on this interval is more than 
$$
N\rho n^{-(3/2+o(1))n} = n^{(5/2 + \eps - n -3/2+o(1))n} = n^{(\eps+o(1))n},
$$
which is larger that $n^{n\eps / 2}$ for large $n$.
\end{proof}

\begin{proof}[Proof of Corollary \ref{a11.1}]
We set  $N =  n^{n^\alpha}$ and $\rho=n^{-\frac13 n^\alpha}$. Since $\alpha >1$, we have $\rho\leq n^{-n}$, and also $\log(N\rho) = \frac23 n^{\alpha} \log n \geq 2n\log n$ for large enough $n$. An application of Theorem \ref{th:Np} yields a family $\calB$ of balls that is $n^{\frac13 n^\alpha}$-connected on some interval $[t_1,t_2]$, and the number of collisions among balls in $\calB$ on this interval is more than 
$$
N\rho n^{-(3/2+o(1))n} = n^{2 n^\alpha/3 -(3/2 +o(1))n}.
$$
The right hand side is larger than $n^{\frac13 n^\alpha}$ for large enough $n$. 
\end{proof}

\section{Acknowledgments}
We are grateful to Jayadev Athreya, Sara Billey, Branko Gr\"unbaum, Jaime San Martin and Rekha Thomas for very helpful advice.

\bibliographystyle{alpha}
\bibliography{hard}

\end{document}